\documentclass[]{interact}

\usepackage{epstopdf}
\usepackage[caption=false]{subfig}
\usepackage{xcolor}

\usepackage[numbers,sort&compress]{natbib}
\bibpunct[, ]{[}{]}{,}{n}{,}{,}
\makeatletter
\def\NAT@def@citea{\def\@citea{\NAT@separator}}
\makeatother

\theoremstyle{plain}
\newtheorem{theorem}{Theorem}[section]
\newtheorem{lemma}[theorem]{Lemma}

\newtheorem{proposition}[theorem]{Proposition}
\theoremstyle{definition}
\newtheorem{definition}[theorem]{Definition}
\newtheorem{example}[theorem]{Example}

\theoremstyle{remark}
\newtheorem{remark}{Remark}

\newtheorem{asumption}[theorem]{Assumption}

\begin{document}


\title{Lagrange Multipliers, Duality, and Sensitivity in Set-Valued Convex Programming via Pointed Closed Convex Processes}

\author{
\name{Fernando Garc\'ia-Casta\~no\textsuperscript{a}\thanks{CONTACT Fernando Garc\'ia-Casta\~no  Email: Fernando.gc@ua.es} and M.A.  Melguizo Padial\textsuperscript{b}}
\affil{Department of Applied Mathematics, University of Alicante, Alicante, Spain; \textsuperscript{a}Orcid: 0000-0002-8352-8235;\textsuperscript{b}Orcid: 0000-0003-0303-791X}
}

\maketitle

\begin{abstract}
We present a new kind of Lagrangian duality theory for set-valued convex optimization problems whose objective and constraint maps are defined between preordered normed spaces.  The theory is accomplished by introducing a new set-valued Lagrange multiplier theorem and a dual program with variables that are pointed closed convex processes.  The pointed nature assumed for the processes is essential for the derivation of the main results presented in this research.  We also develop a strong duality theorem that guarantees the existence of dual solutions,  which are closely related to the sensitivity of the primal program.   It allows extending the common methods used in the study of scalar programs to the set-valued vector case.
\end{abstract}


\begin{amscode}
 90C29, 90C25, 90C31, 90C48
\end{amscode}

\maketitle

\section{Introduction} 
Set-valued optimization is an expanding branch in applied mathematics that has attracted a great deal of attention in the last decades \cite{aubin1990set,Jahn2004,Zalinescu2015,Mordukhovich2018,Pardalos2010}. This topic tackles optimization problems where the objective and/or the constraint maps are set-valued ones acting between abstract spaces.  Set-valued optimization problems have been analysed according to different concepts of solution. Such optimal solutions have usually been defined by means of vector, set,  or lattice approaches. 

In the vector approach, the first to be developed, the concept of solution is based on the standard notion of Pareto minimal point and its many variants.  This approach has been widely developed in the convex case and the corresponding literature is extensive. We refer the reader to \cite{Grad2015,Jahn2004,Zalinescu2015,Luc1988} and the references therein.

However, solution concepts based on vector approaches are sometimes improper. In order to avoid this drawback, it is sometimes convenient to work with order relation for sets.  The solution concept based on set approaches is based on a set order relation. This is obtained by extending the original preordered image vector space to its power set. This approach was introduced by Kuroiwa, see \cite{Kuroiwa1996,Kuroiwa1997,Kuroiwa1998,Kuroiwa2001,Kuroiwa2003}. Important contributions were made later, see for example \cite{HerRoMa2007_1,HerRoMa2007_2,HerRoMa2007_3,HerRoMa2007_4,HerRoMa2009_1}.

Finally, in lattice approaches, the concept of solution is based on a lattice structure on the power set of the image space. The corresponding infimum and supremum are sets related to the sets of weak optimal points. This view is useful for applications of set-valued approaches in the theory of vector optimization, especially in duality theory. We refer the reader to \cite{Lhone2011} for a comprehensive discussion and \cite{Hamel2014} for some subsequent extensions.

In this paper, we deal with the solution concept based on a vector approach. From this perspective, we establish a new Lagrange multiplier theorem. Then, we introduce a new kind of Lagrangian duality scheme for set-valued optimization programs whose objective and constraint maps are convex and defined between preordered normed spaces.   We prove that the sensitivity of the primal program is closely related to the set of dual solutions.  Although such solutions are usually continuous linear operators, the dual variables considered in this paper are pointed closed convex processes.  The use of processes as dual variables  is not entirely new in set-valued analysis. For instance,   we find dual variables that are closed convex processes in \cite{Hamel2009,Hamel2011,Hamel2014,Lhone2011}. The processes in this work are also pointed.  This property seems to improve the adaptability of the variables to the structure of convex set-valued vector problems.  On the other hand, the arguments in this manuscript are direct and broadly geometric.   Roughly speaking, we extend the methods used in the study of scalar programs in \cite[Chapter 8]{Luenberger1969} to the study of set-valued vector programs.  In this way, our work includes the scalar case as a particular one. Finally,  some results in our earlier paper \cite{GC-Melguizo2019} have been enhanced by those obtained in this work. 

The paper is organized as follows. In Section  \ref{Preliminaries and Notations} we state some preliminary terminology. In Section \ref{The Set-valued Lagrange Multipliers} we introduce the parametric constrained set-optimization problem $(P(z))$ to be analysed in the paper. Then we state a new Lagrange multiplier theorem for $(P(0))$. Such a result, Theorem \ref{ThLagMult}, is stated in terms of nondominated points instead of minimal points. Hence, we do not force the primal program to reach its optimal points. In Section \ref{Duality} we define a dual program. Theorem \ref{Tdualidad} guarantees the existence of a dual solution even if we do not assume the existence of a minimal solution in the primal program.  Section \ref{Sensitivity Analysis} is devoted to the study of the sensitivity of the primal program. The formulas for sensitivity in Theorems~\ref{TeorConting} and \ref{thsensi} are expressed in terms of the Lagrange process introduced in \cite{GC-Melguizo2019}. In Remark~\ref{Remark_sensibilidad} we note that the sensitivity of the program is closely related to the solutions of the dual program. Finally, in Section \ref{Conclusions}, we present conclusions that summarize this work and we pose some open problems for further research.

\section{Preliminaries and Notation}\label{Preliminaries and Notations}
Let $Y$ and $Z$ be normed spaces with topological duals $Y^*$ and $Z^*$. We denote by $\mbox{cl}(A) $, $\mbox{int}(A) $, $\mbox{bd}(A)$, $\mbox{co}(A)$, and $\mbox{cone}(A)$ the closure, interior, boundary,  convex hull, and cone hull of a set $A\subset Y$, respectively. Sometimes the parentheses will be omitted. We will consider the sum of two subsets in $Y$ in the usual way adopting the convention $A+\emptyset=\emptyset+A=\emptyset$, for every subset $A \subset Y$. A non-empty subset $K\subset Y$ is called a cone if $\alpha K \subset K$, for all $\alpha \in \mathbb{R}_+$, with $\mathbb{R}_+$ the set of non-negative real numbers.  A cone $K\subset Y$ is said to be pointed (resp. solid, proper) if $K \cap (-K)=\{ 0\}$ (resp. int$(K)\not = \emptyset$, $\{0\}\not = K\not =Y$).  Given a set-valued map  $F: Z \rightrightarrows Y$ we may identify, in a natural way, $F$ with its graph which is the set defined by Graph$(F):=\{(z,y)\in Z\times Y \colon y\in F(z)\}$. The domain of $F$ is defined by Dom$(F):=\{z\in Z: F(z) \neq \emptyset\}$ and the image of $F$ is defined by Im$(F):=\cup_{z \in \mbox{Dom}(F)} F(z)$.  The image by $F$ of the set $A\subset \mbox{Dom}(F)$ is $F(A):=\cup_{a \in A}F(a)$. 
A set-valued map $F$ is said to be a process if Graph$(F)$ is a cone. A process $F$ is said to be  convex (resp. closed, pointed) if Graph$(F)$  is convex (resp. closed, pointed). We will denote by $P(Z,Y)$ the set of all closed convex processes $F$ from $Z$ into $Y$ such that Dom$(F)=Z$.  Let $Y_+\subset Y$ be a convex cone and pick arbitrary $y_1$, $y_2\in Y$, we write $y_1\leq y_2$ if and only if $y_2-y_1\in Y_+$. Then ``$\leq$''  defines a reflexive and transitive  relation (a preorder) on $Y$ and the cone $Y_+$ is called the ordering cone on $Y$.  A set-valued map $F: C \rightrightarrows Y$defined on a non-empty convex subset $C \subset Z$, is said to be $Y_+$-convex if its epigraph, Epi$(F):=\left\lbrace (z,y)\in Z \times Y : y \in F(z)+Y_+  \right\rbrace $, is convex.  We say that a point $y_0\in Y$ is nondominated by the set $A\subset Y$,  if $A  \cap (y_0-Y_+) \subset y_0+Y_+$. We say that a point $y_0 \in Y$ is a minimal point of a set $A\subset Y$, written $y_0\in$ Min$(A)$, if $y_0\in A$ and $y_0$ is nondominated by $A$.If $Y_+$ is pointed, then $y_0\in$ Min$(A)$ if and only if $A  \cap (y_0-Y_+) =\{y_0\}$. Let us note that in the real line with the usual order minimal points become minima, and any nondominated point belonging to the closure of a set becomes its infimum. Analogously, we say that  $y_0 \in A$ is a maximal point of a set $A\subset Y$, written $y_0\in$ Max$(A)$, if  $A  \cap (y_0+Y_+) \subset y_0-Y_+$. If $Y_+$ is pointed, then $y_0\in$ Max$(A)$ if and only if $A  \cap (y_0+Y_+) =\{y_0\}$.  If the ordering cone $Y_+$ is solid, then we can introduce the following relation. For arbitrary $y_1$, $y_2\in Y$ we write  $y_1 < y_2$ if and only if $y_2-y_1\in$ $\mbox{int}(Y_+)$.   We say that  a point $y_0\in A$ is a weak minimal (resp. weak maximal) point of a set $A\subset Y$, written $y_0\in$ WMin$(A)$ (resp.  $y_0\in$ WMax$(A)$),  if $A  \cap (y_0- \mbox{int}(Y_+)) = \emptyset$ (resp. $A  \cap (y_0+ \mbox{int}(Y_+)) = \emptyset$ ).

\section{Problem Formulation and Set-Valued Lagrange Multipliers}\label{The Set-valued Lagrange Multipliers}
We begin this section by establishing the set-valued vector constrained problem $(P(z))$ which will be analysed throughout this paper. We next introduce the notion of nondominated point of $(P(z))$. Then we introduce the notion of Lagrange multiplier of $(P(0))$ associated with a nondominated point, and we show that such Lagrange multipliers always exist. We associate to each  nondominated point $y_0$ of $(P(0))$, a non-empty set   $\Gamma_{y_0}$ of pointed closed convex processes. The main result of the section, Theorem~\ref{ThLagMult},  assures that every element in $\Gamma_{y_0}$ is a Lagrange multiplier of $(P(0))$.

Here and subsequently we consider $X$, $Y$, and $Z$ normed spaces.  We assume that $Y$ and $Z$ are preordered normed spaces and that each corresponding ordering cone, $Y_+$ and $Z_+$, is proper and solid.  Let $\Omega \subset X$ be a convex set.  Let us be given the set-valued maps $F:\Omega \rightrightarrows Y$ and $G:\Omega \rightrightarrows Z$   where $F$ is  $Y_+$-convex and $G$ is $Z_+$-convex.  These maps determine the following parametric set-optimization problem:
\[
\text{  Min \,} F(x) \text{ \  such that \ } x\in \Omega, \text{ \ } G(x) \cap \left( z-Z_+ \right) \neq \emptyset. \text{ \ \ \ \ \ \ } (P(z))
\] 
We adopt the following notations: for every $z \in Z$ we have  $S(z):=\{x \in \Omega\colon  G(x) \cap \left( z-Z_+ \right) \neq \emptyset \}$,  $V (z):= F(S(z))$, $(V + Y_+) (z):= V (z)  + Y_+ $,  and
$M(z):= \text{Min}(V(z)$). With these notations, we obtain in addition four set-valued maps:
$ S : Z\rightrightarrows X$,  $V: Z\rightrightarrows Y$,  $V + Y_+ : Z\rightrightarrows Y$, and $M: Z\rightrightarrows Y$. In the terminology of scalar optimization, the map $M$ is called the marginal function.  In our setting, we will refer to $M$ as the set-valued marginal map. An important question in connection with parametric optimization is the study of the derivative  of such a map. Section \ref{Sensitivity Analysis} is devoted to such a topic (sensitivity).

\begin{definition} \label{InfumablePoint}
Let $z \in Z$ and consider the corresponding program $(P(z))$.
\item[(i)] A pair $(x_z,y_z)\in S(z)\times V(z)$  is called a minimizer  of  $(P(z))$ if $y_z \in F(x_z) \cap  M(z)$. Then $y_z$ is said to be a minimal point of $(P(z))$. 
\item[(ii)] A point $y_z\in Y$ is called a nondominated point of  $(P(z))$, written $y_z\in ND(P(z))$,  if  $y_z\in \mbox{cl}(V(z))$ and  $y_z$ is nondominated by the set $V(z)$.
\end{definition}

\begin{remark}
The notion of nondominated point is less restrictive than that of minimal point because a nondominated point is not required to be achieved at a feasible solution. This fact is significant because the optima of the problem can be chosen among the corresponding nondominated points.
\end{remark}
The constraint in $(P(z))$ can be rewritten as, $z \in G(x)+ Z_+$, which leads us to the next notion.
\begin{definition} \label{LMs}
Let $y_0\in ND(P(0))$. A closed convex process $\Delta: Z\rightrightarrows Y$ is said to be a Lagrange multiplier of $(P(0))$ at $y_0$, if $y_0$ is a nondominated point of the program
\[
\text{  Min \,} F(x)+\Delta(G(x)+Z_+) \text{ \  such that \ } x\in \Omega.\text{ \ \ \ \ \ \ } (P[\Delta])
\]
In other words, if $y_0\in \mbox{cl}(\cup_{x \in \Omega}(F(x)+\Delta(G(x)+Z_+)))$ and $y_0$ is nondominated by the set $\cup_{x \in \Omega}(F(x)+\Delta(G(x)+Z_+))$.
\end{definition}

\begin{definition}\label{Definicion_S_Y}
Let $y_0\in ND(P(0))$.  We define
\begin{equation*}\label{defi_S_Y}
\begin{array}{ll}
\mathcal{S}_{Y_+}(y_0):= & \left\lbrace  (z^*, y^*)\in Z^* \times Y^* : \langle z^*, z'\rangle +\langle y^*, y'\rangle  \leq  \langle y^*, y_0 \rangle \leq \langle z^*, z\rangle +\langle y^*, y\rangle ,\right.  \\
             & \left. \forall (z', y') \in (-Z_+)\times (y_0-Y_+), \forall (z,y)\in \mbox{Graph}(V +Y_+) \right\rbrace,
\end{array}
\end{equation*}
and $\mathcal{S}'_{Y_+}(y_0):=\mathcal{S}_{Y_+}(y_0)\setminus\left\lbrace (0,0)\right\rbrace.$ 
\end{definition}
\begin{remark}
The set $\mathcal{S}_{Y_+}(y_0)$ is a convex cone and contains at least a non-zero element. That is a consequence of Eidelheit's separation theorem (see e.g. \cite[Theorem 1.1.3, p. 133]{Luenberger1969}), which can be applied because $\mbox{int}((-Z_+)\times (y_0-Y_+)) \neq \emptyset$, the set $ \mbox{Graph}(V+Y_+)$ is convex, and $(0,y_0) \in \mbox{bd}(\mbox{Graph}(V+Y_+))$.  To check the convexity of $\mbox{Graph}(V+Y_+)$  we consider the set-valued maps
$(G, F):\Omega \rightrightarrows Z \times Y$ and $((G, F)+(Z_+ \times Y_+)):\Omega \rightrightarrows Z \times Y$ defined respectively  by $(G,F)(x):=(G(x),F(x))$ and $((G, F)+(Z_+ \times Y_+))(x):= (G(x)+Z_+,F(x)+Y_+)$, $\forall x \in \Omega$.  Then, such a convexity follows directly from the $(Z_+ \times Y_+)$-convexity of the set-valued map $(G, F)$ and the equality $\mbox{Graph}(V+Y_+)= \mbox{Im}((G, F)+(Z_+ \times Y_+) )$. The choice $y_0\in ND(P(0))$ implies $(0,y_0) \in \mbox{bd}(\mbox{Graph}(V+Y_+))$.  Indeed, if $(0,y_0) \in \mbox{int}(\mbox{Graph}(V+Y_+))$, then  there would exist  $y_+ \in Y_+ \setminus \{0\} $ such that $(0, y_0-y_+) \in \mbox{Graph}(V+Y_+)$. As a consequence, we would have $x\in \Omega $ and $y \in F(x)  $ such that $G(x)\cap (-Z_+)\neq \emptyset$ and $y\leq y_0-y_+ < y_0$. That contradicts $y_0\in ND(P(0))$.
\end{remark}
From now on and throughout the whole paper, we assume the usual Slater constraint qualification. It will appear in the statements of the main results in the paper  guaranteeing that $y^* \not = 0 $  for every $(z^*, y^*) \in \mathcal{S}'_{Y_+}(y_0)$.

\begin{asumption}[Slater constraint qualification]\label{HipoSlater}
There exists $x_1 \in \Omega$ for which $G(x_1) \cap \left(  - \text{int}(Z_+) \right)\neq \emptyset$.
\end{asumption}

In what follows, we will make use of the set-valued maps introduced by Hamel in \cite{Hamel2009}.  For every $(z^*,y^*) \in Z^* \times Y^* \setminus \{(0,0)\} $, the set-valued map $S_{(z^*,y^*)}: Z \rightrightarrows Y$ is defined by $S_{(z^*,y^*)}(z):=\left\lbrace y\in Y :   \langle z^*,z\rangle + \langle y^*,y \rangle \leq 0 \right\rbrace, \, \forall z \in Z$.  

\begin{definition}\label{DefGammay0}
Let $y_0\in ND(P(0))$ and  $(z^*,y^*) \in  \mathcal{S}'_{Y_+}(y_0)$.  We define 
\item[(i)] The set of processes associated to $(z^*,y^*)$ by \[\Lambda_{(z^*,y^*)}:=\{\Delta \in P(Z,Y)\colon \text{Graph}(\Delta)\setminus\left\lbrace (0,0)\right\rbrace \subseteq \mbox{int} (\text{Graph}(-S_{(z^*,y^*)}))\}.\]
\item[(ii)] The set of processes associated to $y_0$ by
\begin{equation}\label{Formula_defi_Gamma_y0}
\Gamma_{y_0}:= \bigcup_{(z^*,y^*) \in  \mathcal{S}'_{Y_+}(y_0)}{\Lambda_{(z^*,y^*)}} \subset P(Z,Y).
\end{equation}
\end{definition}
Lemma \ref{Lemma_usamos_Hipo_Slater} below states that assumption \ref{HipoSlater} implies $\Gamma_{y_0}\not = \emptyset$. Then we are guaranteed that such a set is non-empty  for every nondominated point of $(P(0))$.

\begin{lemma}\label{Lemma_Delta_pointed}
Let $y_0\in ND(P(0))$ and $\Delta \in \Gamma_{y_0}$. Then $\Delta$ is a closed, convex, and pointed process.
\end{lemma}
\begin{proof}
Let us fix $\Delta \in \Gamma_{y_0}$. We will check that Graph($\Delta$) is a pointed cone. Indeed, for $\Delta \in \Gamma_{y_0}$ there exists $(z^*,y^*) \in  \mathcal{S}'_{Y_+}(y_0)$ such that $\text{Graph}(\Delta)\setminus\left\lbrace (0,0)\right\rbrace \subseteq \mbox{int} (\text{Graph}(-S_{(z^*,y^*)}))=\{(z,y)\in Z \times Y\colon -z^*(z)+y^*(y)>0\}$. Now, assume that $(z,y) \in \text{Graph}(\Delta)\cap (-\text{Graph}(\Delta))$. If $(z,y)\not = (0,0)$, then 
$-z^*(z)+y^*(y)>0$ and $-z^*(-z)+y^*(-y)=z^*(z)-y^*(y)>0$, a contradiction. Thus $(z,y) = (0,0)$.
\end{proof}

\begin{lemma}\label{Lemma_cono}
Let $X$ be a normed space, $B \subset X$ the closed unit ball, $T \in X^*$, and $x_0 \in X$. If $T(x_0)>0$, then there exits $\epsilon>0$ such that ${\mbox{cl}(\mbox{cone}}(x_0+\delta B))\setminus\{0\}\subset \{x \in X\colon T(x)>0\}$, for every $0<\delta\leq \epsilon$.
\end{lemma}
\begin{proof}
Let $\epsilon>0$ be such that $T(x_0+u)>0$, $\forall u \in \epsilon B$, and $0 \not \in x_0+\epsilon B$. Clearly we have that  ${\mbox{cone}}(x_0+\epsilon B)\setminus\{0\}\subset \{x \in X\colon T(x)>0\}$. We will check that ${\mbox{cl}(\mbox{cone}}(x_0+\epsilon B))={\mbox{cone}}(x_0+\epsilon B)$. Conversely, suppose that there exists $x \in \mbox{cl}(\mbox{cone}(x_0+ \epsilon B)) \setminus \mbox{cone}(x_0+\epsilon B)$. Then $0\not =x =\lim_n \lambda_n(x_0+u_n)$, for some sequence $(\lambda_n)_n$ of positive real numbers and some sequence $(u_n)_n \subset \epsilon B$. The condition $0 \not \in x_0+\epsilon B$ assures that the sequence $(\lambda_n)_n$ is bounded. Otherwise, we would have (maybe for some subsequence) that $0=\lim_n \frac{x}{\lambda_n}=\lim_n (x_0+u_n)$, a contradiction. So, it is not restrictive to assume that $\lim_n\lambda_n=\lambda\geq 0$. If $\lambda>0$, then $\frac{x}{\lambda}-x_0=\lim_n \frac{x}{\lambda_n}-x_0= \lim_n u_n \in \epsilon B$. As a consequence, there exists $u \in \epsilon B$ such that $x=\lambda(x_0+u) \in \mbox{cone}(x_0+\epsilon B)$, a contradiction. Therefore $\lim_n\lambda_n= 0$.  Since $B$ is bounded,  it follows that $x =\lim_n \lambda_n(x_0+u_n)=0$. Again a contradiction. The same reasoning applies to the case $0<\delta\leq \epsilon$.
\end{proof}

\begin{lemma}\label{Lemma_usamos_Hipo_Slater}
Let $y_0\in ND(P(0))$ and assume that there exists $x_1 \in \Omega$ for which $G(x_1) \cap \left(  - \text{int}(Z_+) \right)\neq \emptyset $.   Then $\Lambda_{(z^*,y^*)} \neq \emptyset, \forall (z^*,y^*) \in  \mathcal{S}'_{Y_+}(y_0)$.
\end{lemma}
\begin{proof}
Let us fix $(z^*,y^*) \in  \mathcal{S}'_{Y_+}(y_0)$.
The  Slater constraint qualification and the equality $\mbox{int} (\mbox{Graph}(-S_{(z^*,y^*)}))=\left\lbrace (z,y)\in Z \times Y :   \langle -z^*,z\rangle + \langle y^*,y \rangle > 0 \right\rbrace$ imply that $(0,y_+) \in  \mbox{int} (\mbox{Graph}(-S_{(z^*,y^*)}))$, $\forall y_+ \in \mbox{int} (Y_+)$. Indeed, since the elements of $S_{Y_+}(y_0)$ take non-negative values on $Z_+\times Y_+$,  it follows that $y^*(y_+)\geq 0$, $\forall y_+\in Y_+$. Now, the Slater constraint qualification guarantees that $y^*(y_+)> 0$,  $\forall y_+\in \mbox{int}(Y_+)$.

Next, we define a process  $\Delta \in P(Z,Y)$ such that  Graph$(\Delta)\setminus\left\lbrace (0,0)\right\rbrace \subseteq \mbox{int} (\text{Graph}(-S_{(z^*,y^*)}))$. We fix $y_+ \in \mbox{int} (Y_+)$ and take $\epsilon>0$ from Lemma~\ref{Lemma_cono} corresponding to $T=(-z^*,y^*) \in (Z\times Y)^*$ and $(0,y_+)\in Z \times Y$. Let $B$ the closed unit ball of $Z\times Y$ and $0<\delta \leq \epsilon$ such that $(0,y_+) + \delta B \subseteq \mbox{int} (\mbox{Graph}(-S_{(z^*,y^*)}))$ and $(0,0)\not \in (0,y_+) + \delta B$. Define the closed cone $K:=\mbox{cl}(\mbox{cone}((0,y_+) + \delta B))\subset Z \times Y$  and the closed process $\Delta\in P(Z,Y)$ as Graph($\Delta$):=$K$. Let us check that $Dom(\Delta)=Z$. Fix any $z \in Z$. As $0$ belongs to interior of the canonical projection  of  $\delta B$ on $Z$, there exists some $n$ such that $z/n$ belongs to such a projection. Then, there exists $y \in Y$ such that $(z/n,y)\in \delta B$. Hence $(0,ny_+)+(z,ny)\in K$, which yields $n(y+y_+)\in \Delta(z)$. To finish, we note that adapting the proof of Lemma \ref{Lemma_cono} for $T=(-z^*,y^*) \in (Z\times Y)^*$, we have that $K\setminus \{(0,0)\}\subset \mbox{int} (\mbox{Graph}(-S_{(z^*,y^*)}))$. 
\end{proof}

According to the natural order in $\mathbb{R}$, the infimum of a lower bounded set of real numbers can be seen as a  nondominated point of its closure. So, our next result contains the classical scalar theorem \cite[Theorem~1, p.~217]{Luenberger1969} as a particular case.

\begin{theorem} \label{ThLagMult}
Let $y_0\in ND(P(0))$ and $\Delta \in \Gamma_{y_0}$. Assume that there exists $x_1 \in \Omega$ for which $G(x_1) \cap \left(  - \text{int}(Z_+) \right)\neq \emptyset$.  Then $\Delta$ is a Lagrange multiplier of $(P(0))$ at $y_0$, i.e.,  
$y_0$ is a nondominated point of the program
\begin{equation*}\label{DualPr}
\text{  Min \,} F(x)+\Delta(G(x)+Z_+) \text{ \  such that \ } x\in \Omega.  \text{ \ \ \ \ \ \ } (P[\Delta])
\end{equation*}
Furthermore, if $y_0$ is a minimal point of $(P(0))$, that is, if  $y_0 \in F(x_0)$ for a feasible solution $x_0$, then $y_0$ is a minimal point of  $(P[\Delta]) $  also achieved at $x_0$ and 
\begin{equation}\label{InterseccionTeo}
\Delta(G(x_0)+Z_+)\cap (-Y_+) \subseteq (-Y_+)\cap  Y_+.
\end{equation}

\end{theorem}
\begin{proof}
Let us prove that $\Delta \in \Gamma_{y_0}$ is a Lagrange multiplier of $(P(0))$ at $y_0$. We first check that $y_0 \in \mbox{cl}(\cup_{x \in \Omega}(F(x)+\Delta(G(x)+Z_+)))$. Let $B\subset Y$ be an arbitrary ball centred at the origin. Since $(y_0+B)\cap V(0) \not = \emptyset$, we can choose $u \in B$, $x_0 \in \Omega$ such that $G(x_0)\cap (-Z_+)\not = \emptyset$, and $y_1 \in F(x_0)$ such that $y_0+u =y_1$. Now, since  $G(x_0)\cap (-Z_+)\not = \emptyset$ is equivalent to $0 \in G(x_0)+Z_+$, we have that  $y_0+u =y_1+0 \in (y_0+B)\cap (F(x_0)+\Delta(G(x_0)+Z_+))$. Next, we will check that $y_0$ is nondominated by the set $\cup_{x \in \Omega}(F(x)+\Delta(G(x)+Z_+))$. On the contrary, suppose that there exist $\bar{x} \in \Omega$, $\bar{y}_1 \in F(\bar{x})$, and $\bar{y}_2 \in \Delta(\bar{z}_1 + \bar{z}_2)$ such that $\bar{z}_1 \in G(\bar{x})$, $ \bar{z}_2 \in Z_+$,  and $\bar{y}_1 + \bar{y}_2 \in y_0 - Y_+$. Fix $(z^*,y^*) \in  \mathcal{S}'_{Y_+}(y_0)$ such that $\Delta \in \Lambda_{(z^*,y^*)}$. Then $\langle -z^*,\bar{z}_1+\bar{z}_2\rangle+\langle  y^*, \bar{y}_2\rangle >0$ because $(\bar{z}_1+\bar{z}_2,\bar{y}_2) \in \mbox{Graph} (\Delta)\setminus \{(0,0)\}\subset \mbox{int} (\mbox{Graph}(-S_{(z^*,y^*)}))$. On the other hand, since $\bar{z}_1 \in G(\bar{x}) \subset G(\bar{x})+Z_+$, we have $G(\bar{x})\cap (\bar{z}_1-Z_+)\not=\emptyset$. Hence $(\bar{z}_1,\bar{y}_1)\in \mbox{Graph}(V)$. Analogously, since $\bar{z}_1+\bar{z}_2 \in G(\bar{x})+Z_+$, we have $(\bar{z}_1+\bar{z}_2,\bar{y}_1)\in \mbox{Graph}(V)\subset  \mbox{Graph}(V+Y_+)$. As a consequence, by definition of $S_{Y_+}(y_0)$, we have $\langle y^*,y_0\rangle\leq \langle z^*,\bar{z}_1+\bar{z}_2\rangle+\langle y^*,\bar{y}_1\rangle$. Therefore, if we define $\hat{y}:= \bar{y}_1 + \bar{y}_2-y_0 \in -Y_+$, then we obtain
$\langle y^*,\hat{y}\rangle = \langle z^*, (\bar{z}_1+\bar{z}_2)-(\bar{z}_1+\bar{z}_2)\rangle +\langle y^*, \bar{y}_1 + \bar{y}_2-y_0\rangle =
\langle z^*, \bar{z}_1+\bar{z}_2\rangle+ \langle y^*, \bar{y}_1 \rangle- \langle y^*, y_0 \rangle + \langle -z^*, \bar{z}_1+\bar{z}_2\rangle+ \langle y^*, \bar{y}_2 \rangle>0$. But $ \hat{y} \in -Y_+$  and $y^* \geq 0$ on $Y_+$, a contradiction. Hence  $y_0$ is  a nondominated point of $(P[\Delta])$ and $\Delta $ is a Lagrange multiplier of $(P(0))$ at  $y_0$.

To check the second sentence in the statement, we assume that $y_0$ is a minimal point of $(P(0))$ and $y_0 \in F(x_0)$ for some $x_0 \in \Omega$ such that $G(x_0)\cap (-Z_+)\not = \emptyset$. As we proved above,  the process $\Delta $ is a Lagrange multiplier of $(P(0))$ at  $y_0$. Furthermore, $y_0=y_0+0 \in F(x_0)+\Delta(G(x_0)+Z_+)$, and so, $y_0$ is a minimal point of $(P[\Delta])$ achieved at $x_0$.

Let us finish showing (\ref{InterseccionTeo}). Let $u \in \Delta(G(x_0)+Z_+)\cap (-Y_+)$. Since $y_0$ is a minimal point of $(P[\Delta])$ achieved at $x_0$, we have that $y_0 \in F(x_0)$ and $\left( F(x_0)+ \Delta(G(x_0)+Z_+) \right) \cap  (y_0-Y_+)\subseteq \left\lbrace y_0 \right\rbrace + Y_+$. Thus $y_0+u \in \{y_0\}+Y_+$ and $u \in  (-Y_+)\cap Y_+$.
\end{proof}

\begin{remark} \label{IntersectionTeo1}
In the proof of inclusion (\ref{InterseccionTeo}), we only assume that $\Delta$ is a Lagrange multiplier. Therefore, such an inclusion still holds if we consider Lagrange multipliers of $(P(0))$ at $y_0$ which do not  belong to $\Gamma_{y_0}$.
\end{remark}

Next, we adapt Example 4.4 in \cite{GC-Melguizo2019} to our current set-valued  context to check that the proper inclusion $\text{Graph}(\Delta)\setminus\left\lbrace (0,0)\right\rbrace \subseteq \text{int} (\text{Graph}(-S_{(z^*,y^*)}))$ from Definition \ref{DefGammay0} becomes decisive to detect the optimal point in $(P[\Delta])$. The inclusion is directly related to the property that $\Delta$ is pointed. Let us note that additional requirement is assumed neither on the optimal point (such as some type of proper efficiency) nor on the ordering cone.

\begin{example}\label{Example}
Let $X=Y=\mathbb{R}^2$, $Y_+=\left\lbrace (x_1,x_2) \in \mathbb{R}^2 : x_1 \geq 0, x_2 \geq 0  \right\rbrace$, $Z=\mathbb{R}$, $Z_+=\mathbb{R}_+$, $\Omega= \left\lbrace (x_1,x_2) \in \mathbb{R}^2 : x_2 > 0  \right\rbrace \cup \lbrace (0,0) \rbrace$, $F(x_1,x_2)=\{(x_1,x_2)\}$ , and $G(x_1,x_2)=\{x_2-1 \}$. 

Then $y_0=(0,0)$ is a minimal point of ($P(0)$). Indeed,  for $z \geq -1$ we have $S(z)=\left\lbrace (x_1,x_2) \in \mathbb{R}^2 : 0 < x_2 \leq 1+z  \right\rbrace \cup \lbrace (0,0) \rbrace$,  and for $z<-1$ we have $S(z)=\emptyset$. Besides $V(z)=F(S(z))=S(z)$, $\forall z \in \mathbb{R}$.  Define the set $A:=\lbrace (x_1,x_2) \in \mathbb{R}^2 : 0 \leq x_1, 0 \leq x_2 \rbrace \cup \lbrace (x_1,x_2) \in \mathbb{R}^2 : x_1 < 0 < x_2 \rbrace$. For $z\geq -1$ we have $(V+Y_+)(z)= V(z)+Y_+ =A$, and for $z<-1$ we have $(V+Y_+)(z)=\emptyset$.  Therefore, for $z\geq -1$ we have 
$\mbox{Min}(\left\lbrace F(x): x \in \Omega, \,z \in  G(x)+ Z_+ \right\rbrace)= \lbrace (0,0) \rbrace$, and for 
$z<-1$ we have $\mbox{Min}(\left\lbrace F(x): x \in \Omega, \,z \in  G(x)+ Z_+ \right\rbrace)=\emptyset$.
 
Now, consider the particular program ($P(0)$) and $(0,y_0)=(0,(0,0))$. Then Graph$(V+Y_+)= [-1, \infty) \times A \subset \mathbb{R}^3$, $ S_{Y_+}(0,0)=\left\lbrace (0,0,\lambda):\lambda > 0 \right\rbrace$, and $-S_{(z^*,y^*)}(z)=\mathbb{R}_+$, $\forall (z^*, y^*) \in S_{Y_+}(0,0)$. We take $\Delta_{(0,0)} \in P(\mathbb{R},\mathbb{R}^2)$ such that 
$\mbox{Graph}(\Delta_{(0,0)})=\left\lbrace(z,(y_1,y_2))): | y_2 | \geq \max\{|z|, | y_1 |\} \right\rbrace. $
Then, for $x_2 \leq 1$ we have $\Delta_{(0,0)}(G(x_1,x_2)+Z_+) = \Delta_{(0,0)}(x_2-1+\mathbb{R}_+) =\left\lbrace (x'_1,x'_2) \in \mathbb{R}^2 : x'_2 \geq | x'_1 |  \right\rbrace$, and for $x_2>1$ we have $\Delta_{(0,0)}(G(x_1,x_2)+Z_+) = \Delta_{(0,0)}(x_2-1+\mathbb{R}_+)=\left\lbrace (x'_1,x'_2) \in \mathbb{R}^2 : x'_2 \geq | x'_1 | \, , x'_2 \geq x_2-1 \right\rbrace$. Clearly $\bigcup_{x \in \Omega}{(F(x)+\Delta(G(x)+Z_+))}=A$. Hence $y_0=(0,0)$ is a minimal point of ($P[\Delta_{(0,0)}]$).
Note that $-S_{(z^*,y^*)}$ does not fit as a Lagrange multiplier for any $(z^*,y^*) \in S_{Y_+}(0,0) $ because  $\bigcup_{x \in \Omega}(F(x)-S_{(z^*,y^*)}(G(x)+Z_+))= \mathbb{R} \times \mathbb{R}_+$. Therefore  $y_0=(0,0)$ does not become a minimal point of ($P[-S_{(z^*,y^*)}]$). 
\end{example}

\section{Duality} \label{Duality}
This section is devoted to the study of a dual problem for the program $(P(0))$. We develop a geometric duality approach analogous to the scalar case approach. In our setting, we make the pointed closed convex processes play the same role as the linear continuous operators play in the scalar case.  Optimal points of the dual problem are weak maximal points. This is an
interesting feature from a practical point of view because weak maximal points can
be obtained via linear scalarizations.

We define the auxiliary set-valued map
$\Psi: P(Z,Y) \rightrightarrows Y$and the dual set-valued  map $\Phi: P(Z,Y)  \rightrightarrows Y$ respectively by 
$\Psi (\Delta):= \bigcup_{x \in \Omega}( F(x)+\Delta(G(x)+Z_+))$ and $\Phi (\Delta):=  \left\lbrace  y\in  \mbox{cl}(\Psi(\Delta)) : y -Y_+ \cap \Psi(\Delta) \subseteq y+ Y_+ \right\rbrace$, $\forall \Delta \in  P(Z,Y)$.  Then the image of a process $\Delta$ by the dual map is the set of points in the border of $\Psi (\Delta)$ which are nondominated by $\Psi (\Delta)$. Theorem~\ref{ThLagMult} above states that if $ND(P(0))\not = \emptyset$, then $\Phi (\Delta) \neq \emptyset$ for every $\Delta \in  \Gamma_{y_0}$. The dual problem of $(P(0))$ is defined by
\[
\text{  WMax \,} \Phi (\Delta) \text{ \  such that \ } \Delta \in P(Z,Y). \hspace*{3cm}(D(0))
\]

\begin{definition}
Consider the program $(D(0))$. A pair  $(\Delta_0 ,y_0)\in P(Z,Y) \times \Phi(P(Z,Y))$ verifying that $y_0 \in \Phi(\Delta_0) $  is said to be a weak  maximizer of $(D(0))$ if there is no $(\Delta' ,y')\in P(Z,Y) \times \Phi(P(Z,Y))$ with $y' \in \Phi(\Delta')$ and such that $y_0 <y'$. Then $y_0$ is said to be a weak maximal point  of $(D(0))$, written $y_0 \in$ WMax$(D(0))$. 
\end{definition}
The first sentence in the following result is on weak duality. However,  the second one is on strong duality (based on Theorem~\ref{ThLagMult}).

\begin{theorem}\label{Tdualidad}
Let $y_0\in ND(P(0))$. The following statements hold.
\begin{itemize}
\item[(i)]  There is no $y_1 \in \cup_{\Delta \in P(Z,Y)} \Phi(\Delta)$ such that $y_0 < y_1$. In particular, there is no $y_1\in WMax(D(0))$ such that $y_0 < y_1$.
\item[(ii)] Assume that there exists $x_1 \in \Omega$ for which $G(x_1) \cap \left(  - \text{int}(Z_+) \right)\neq \emptyset$. Then, there exists $\Delta_0 \in \Gamma_{y_0} \subset P(Z,Y) $ such that $ y_0 \in \Phi (\Delta_0)$,  that is, $y_0$ is a nondominated point of the program $(P[\Delta_0])$. Furthermore, if  $y_0 \in F(x_0)$ for a feasible solution $x_0$, then the nondominated point of $(P[\Delta_0]) $ is also achieved at $x_0$ and we have $\Delta(G(x_0)+Y_+)\cap (-Y_+) \subseteq Y_+\cap (-Y_+)$.
\end{itemize}
\end{theorem}
\begin{proof}
(i) Let us fix $\Delta \in P(Z,Y)$ and $y_1\in \Phi(\Delta)$. Then $y_1 \in \mbox{cl}(\Psi(\Delta))$ and $y_1$ is nondominated by $\Psi(\Delta)$. On the other hand, since $y_0\in ND(P(0))$, we have that $y_0 \in \mbox{cl}(V(0))$. Assume that $y_0 < y_1$, or equivalently, that $y_0 \in y_1 - \mbox{int}(Y_+)$. Then, we can pick some $y'_0 \in (y_1 - \mbox{int}(Y_+)) \cap V(0)$. Since $y'_0 \in V(0)$, there exists $x'_0 \in \Omega $ such that $y'_0 \in F(x'_0)$ and $0 \in G(x'_0)+Z_+$. Then  $y'_0=y'_0+0 \in F(x'_0)+\Delta(G(x'_0)+Z_+) \subset \Psi(\Delta)$. But $y'_0<y_1$, a contradiction. The second part of (i) is immediate.  Sentence $(ii)$ is a consequence of Theorem~\ref{ThLagMult}.
\end{proof}

\begin{example}\label{ExampleDuality}
In Example \ref{Example} we have 
$ND(P(0))= (-\infty , 0] \times \{0 \}$. Then for every $y_0=(x_0, 0) \in ND(P(0))$, we have 
$S'_{Y_+}(y_0)= \{ (0,0,\lambda) \in \mathbb{R}^3: \lambda >0 \}.$ Therefore
$\mbox{Graph}(-S_{(z^*,y^*)})(z)= \mathbb{R}^2 \times \mathbb{R}_+ $, for every $(z^*,y^*) \in S'_{Y_+}(y_0) $ and $z \in Z$.
Consequently, for every $\Delta \in \Gamma_{y_0}$ we have the inclusion $\mbox{Graph}(\Delta) \setminus \{ (0,0,0) \}\subseteq \mathbb{R}^2 \times (\mathbb{R}_+)\setminus \{0\}$. This yields 
$\Psi (\Delta)= \bigcup_{x \in \Omega}{     \left\lbrace  F(x)+\Delta(G(x)+Z_+) \right\rbrace} = A$, for every $\Delta \in \Gamma_{y_0}$. Therefore, every nondominated point of $(P(0))$ is a nondominated point of $(P[\Delta])$. In this case both sets of nondominated points coincide.
\end{example}

\section{Sensitivity Analysis}\label{Sensitivity Analysis}
In this section, we analyse the sensitivity of $(P(0))$ by means of the contingent derivative.  In our analysis, the concept of Lagrange process plays a crucial role. But we do not use its original definition given in \cite[Definition~3.7]{GC-Melguizo2019}. Instead of that, we introduce it by employing the concept of adjoint process. The adjoint of a process was introduced in \cite{Rockafellar1970}, and it is also known as the transpose of a process \cite[Definition~2.5.1]{aubin1990set}.

From now on, we assume that the ordering cone $Y_+$ on the preordered normed space $Y$ is pointed.  Then, the preorder on Y induced by $Y_+$  becomes an order. 

\begin{definition} Let $P: Z \rightrightarrows Y$ be a process and  $Q : Y^* \rightrightarrows Z^* $  a convex process. 
\item[(i)] The adjoint of $P$ is the process $P^\star : Y^* \rightrightarrows Z^* $ defined by $P^\star (y^*):=\left\lbrace z^* \in Z^* : \langle z^*,z \rangle \leq \langle y^*,y\rangle, \forall
z \in Z, y \in P(z) \right\rbrace$, $\forall y^* \in Y^*$.
\item[(ii)] The adjoint of $Q$ is the process $Q^\star: Z \rightrightarrows Y$  defined by
$Q^\star (z):=\{y \in Y :  \langle z^*,z\rangle \leq \langle y^*,y \rangle,\, \forall
y^* \in Y^*, z^* \in Q(y^*)\}$, $\forall z \in Z$.   
\end{definition}

It is clear that the processes $P^\star$ and $Q^\star$ above are  convex and closed.  

\begin{remark} \label{PrAdjSzy} The adjoint of a process is closely related to the set-valued maps $S_{(z^*,y^*)}$ introduced in Section \ref{The Set-valued Lagrange Multipliers}. Indeed, 
given a convex process $Q: Y^* \rightrightarrows Z^*$ and $z\in Z$, we have $Q^\star (z)=  \{ y \in Y :  \langle z^*,z\rangle + \langle y^*,-y \rangle \leq 0,\,\forall y^* \in Y^*, z^* \in Q(y^*)\}$.
Now, taking into account that
$ -S_{(z^*,y^*)}(z)=\left\lbrace y\in Y : \right.$ $\left.  \langle z^*,z\rangle + \langle y^*,-y \rangle \leq 0 \right\rbrace , $
we get that $Q^\star (z)=  \bigcap_{(y^*,z^*) \in \text{Graph}(Q)} (-S_{(z^*,y^*)}(z))$.

\end{remark}

Let us fix a point $y_0\in \mbox{ND}(P(0))$. We define the process $S_{Y_+}(y_0):Y^* \rightrightarrows Z^*$ by the following abuse of notation. For every $y^*\in Y^*$, the image $S_{Y_+}(y_0)(y^*)$ is the set of $z^* \in Z^*$ such that $(z^*,y^*)$ belongs to the set $S_{Y_+}(y_0)$ in Definition~\ref{Definicion_S_Y}. In this way, we use the same symbol $S_{Y_+}(y_0)$ for this process and the set in Definition~\ref{Definicion_S_Y}. The adjoint that process appears in the following definition.

\begin{definition}\label{defi_LagrangeProcess}
Let $y_0\in ND(P(0))$. We define the Lagrange process of $(P(0))$ at $y_0$ as the closed convex process  $\mathcal{L}_{y_0}:Z\rightrightarrows Y $  such that 
\begin{equation}\label{Ec_defi_LagrangeProcess}
\text{Graph}( \mathcal{L}_{y_0}):= \mathcal{S}_{Y_+}(y_0)^\star=\bigcap_{(z^*,y^*) \in  \mathcal{S}_{Y_+}(y_0)} (-S_{(z^*,y^*)}(z)).
\end{equation}
\end{definition}

\begin{remark}\label{CondicionesAdicionalesLM}
In general, a Lagrange process is not necessarily a Lagrange multiplier for the same program (see \cite[Example 4.4]{GC-Melguizo2019}).  On the other hand,  any condition (a)--(d) in the statement of Theorem \ref{thsensi} below assure that the Lagrange process considered there becomes a Lagrange multiplier.
\end{remark}

Next, we state some terminology on tangent cones and set-valued derivatives (see \cite{aubin1990set} for further details).
\begin{definition}\label{contingentcone}
Let $Y$ be a normed space, $A\subset Y$  a non-empty set, $y\in $ $\mbox{cl}(A)$, and $d$ the metric given by the norm on $Y$. The contingent cone to $A$ relative at $y$, $T_{A}\left( y\right)$,  is the cone defined by $T_{A}\left( y\right) :=\left\{v\in Y\colon \underset{h\rightarrow 0+}{\text{ }\lim
\inf }\ \frac{d\left( A,y+hv\right) }{h}=0\right\}$.
\end{definition}
Next, we introduce the contingent derivative.  Its usual symbol is $D$, and it has not to be confused with that used to denote the dual problem introduced in Section \ref{Duality}.
\begin{definition}
Let $F:Z \rightrightarrows Y$ be a set-valued map and $(z_0,y_0)\in$ Graph$(F)$. The contingent derivative of $F$ at $(z_0,y_0)$ is the set-valued map $DF(z_0,y_0):Z \rightrightarrows Y$ defined by Graph$(DF(z_0,y_0)):=T_{\mbox{Graph(F)}}(z_0,y_0)$. 
\end{definition}

Next, our first result on sensitivity.  The polar cone of a set will be used in the corresponding proof. The negative polar cone of a set $A_1 \subseteq Y$ (resp. $A_2 \subseteq Y^*$) is defined as $A_1^-:=\{ y^* \in Y^*: \langle y^*,y \rangle  \leq 0,\,   \forall  y \in A_1 \}$ (resp.  as $A_2^-:=\{ y \in Y: \langle y^*,y \rangle   \leq 0,\, \forall y^* \in A_2\}$). Their corresponding positive polar cones are defined as $A_1^+:=-A_1^-\subset Y^*$ and $A_2^+:=-A_2^-\subset Y$. Clearly $(A_1^-)^-=(A_1^+)^+$. Note that polar cones of sets in $Y^*$ are sets in $Y$ (not in $Y^{**}$).

\begin{theorem} \label{TeorConting}
Let $y_0$ be a minimal point of $(P(0))$ and $\mathcal{L}_{y_0}$ the Lagrange process of $(P(0))$ at $y_0$. Assume that there exists $x_1 \in \Omega$ for which $G(x_1) \cap \left(  - \text{int}(Z_+) \right)\neq \emptyset$. Then
$\mathcal{L}_{y_0}(-z)= D (V+Y_+)(0,y_0)(z),\, \forall  z \in Z$.

\end{theorem}
\begin{proof}
Let us denote by $\mathcal{L}^-_{y_0}$ the process from $Z$ to $Y$ defined by $\mathcal{L}^-_{y_0}(z):=\mathcal{L}_{y_0}(-z)$, $\forall z \in Z$. Now, by (\ref{Ec_defi_LagrangeProcess}), we have Graph($\mathcal{L}^-_{y_0}$)$=\{(-z,y)\in Z \times Y \colon y \in \mathcal{L}_{y_0}(z)\}= \{(-z,y)\in Z \times Y \colon \langle z^*,z \rangle +\langle y^*,-y\rangle \leq 0, \forall (z^*,y^*)\in \mathcal{S}_{Y_+}(y_0)\}=\{(z,y)\in Z \times Y \colon \langle z^*,z \rangle +\langle y^*,y \rangle \geq 0, \forall (z^*,y^*)\in \mathcal{S}_{Y_+}(y_0)\}=\mathcal{S}_{Y_+}(y_0)^+$.  By \cite[Lemma 3.3 $(iv)$]{GC-Melguizo2019}), we get $ \mathcal{S}_{Y_+}(y_0)=\left( \mbox{Graph}(V+Y_+)-(0, y_0) \right)^+$. Now, using Bipolar Theorem, we have $\mbox{Graph}(\mathcal{L}^-_{y_0})=\left( \mbox{Graph}(V+Y_+)-(0, y_0) \right)^{++}=\mbox{cl}(\mbox{cone} (\mbox{Graph}(V+Y_+)-(0, y_0)))$. On the other hand,  the set $\mbox{Graph}(V+Y_+)$ is convex. Then \cite[Proposition 4.2.1]{aubin1990set}) applies to provide $\mbox{cl}(\mbox{cone} (\mbox{Graph}(V+Y_+)-(0, y_0)))= T_{\mbox{Graph}(V+Y_+)}(0,y_0)$.
Consequently $\mbox{Graph}(\mathcal{L}^-_{y_0})$ coincides with $T_{\mbox{Graph}(V+Y_+)}(0,y_0)$, which directly yields that $\mathcal{L}_{y_0}(-z)= D (V+Y_+)(0,y_0)(z)$, for every $z \in Z$.
\end{proof}


Next, we introduce a derivative  proposed by Shi in \cite{Shi1991}. We will use it to prove the following lemma.
\begin{definition}\label{Defi_derivada_S}
Let $F:Z \rightrightarrows Y$ be a set-valued map and $(z_0,y_0)\in$ Graph$(F)$. The S-derivative of $F$ at $(z_0,y_0)$  is the set-valued map $D_SF(z_0,y_0):Z \rightrightarrows Y$ defined in the following way: for any direction $z \in Z$, the point $y$ belongs to $D_SF(z_0,y_0)(z)$ if and only if there exist two sequences $\{h_{n}\}_{n=1}^{\infty }\subset \mathbb{R}_{+}\backslash \{0\}$ and\ $\{(z_n,y_n)\}_{n=1}^{\infty }\subset Z\times Y$, such that the sequence $\{(z_n,y_n)\}_{n=1}^{\infty}$ converges to $(z,y)$, the sequence $\{h_nz_n\}_{n=1}^{\infty}$ converges to $0$, and $y_0+h_ny_n\in F(z_0+h_nz_n)$,  $\forall n \in \mathbb{N}$.
\end{definition}
Fixed $y_0 \in (V+Y_+)(0)$, it is straightforward to check that 
\begin{equation}\label{Inclusion_Derivada_S}
\text{Graph(}D_S (V+Y_+)(0,y_0))\subseteq \mbox{cl}(\mbox{cone} (\mbox{Graph}(V+Y_+)-(0, y_0))).
\end{equation}

In what follows, we denote by $S_Y$ the unit sphere in $Y$. Let us recall that the marginal set-valued map $M: Z\rightrightarrows Y$ is defined by $M(z):= \text{Min}(V(z)$), $\forall z \in Z$.
\begin{lemma} \label{sens1}
Let $y_0$ be a minimal point of $(P(0))$ and $\mathcal{L}_{y_0}$ the Lagrange process of $(P(0))$ at $y_0$. Assume that the set $Y_+ \cap S_Y $ is compact and that $\mathcal{L}_{y_0}$ is a Lagrange multiplier of $(P(0))$ at $y_0$.
Then $\text{Min} \,  DV(0,y_0)(z)   = \text{Min} \,  D(V+Y_+)(0,y_0)(z), \, \forall z \in Z$.
\end{lemma}
\begin{proof}
We will show the equality $D_S V(0,y_0)(0) \cap (-Y_+)=\lbrace 0 \rbrace$. Then  \cite[Theorem~13.1.1]{Zalinescu2015} applies to the set-valued maps $V$ and $V+Y_+$, providing the desired equality.

First, we check the equality $\mathcal{L}_{y_0}(0)\cap (-Y_+)=\lbrace 0 \rbrace$.  Let us fix an arbitrary $x_0 \in \Omega$ such that $G(x_0)\cap (-Z_+) \neq \emptyset $ and $y_0 \in F(x_0)$.  
The equality $\mathcal{L}_{y_0}(G(x_0))\cap (-Y_+)=\lbrace 0 \rbrace$ is aconsequence of (\ref{InterseccionTeo}), Remark \ref{IntersectionTeo1}, and $Y_+ \cap (-Y_+)=\lbrace 0 \rbrace$.
Now pick  $-v \in \mathcal{L}_{y_0}(0)\cap (-Y_+)$ and $-z_+ \in G(x_0)\cap (-Z_+)$.  Then
$(0,-v) \in  $ Graph($\mathcal{L}_{y_0}$). Now,  by \cite[Lemma 4.1 $(i)$]{GC-Melguizo2019},  we obtain $(-z_+,0) \in$ Graph($\mathcal{L}_{y_0}$).  Since Graph($\mathcal{L}_{y_0}$) is a convex cone, we have  $(0,-v)+(-z_+,0)=(-z_+,-v) \in \text{Graph(}\mathcal{L}_{y_0}),$ that is,  $-v \in \mathcal{L}_{y_0}(-z_+)\subseteq \mathcal{L}_{y_0}(G(x_0))$.  Then the condition $v \neq 0 $ contradicts the equality $\mathcal{L}_{y_0}(G(x_0))\cap (-Y_+)=\lbrace 0 \rbrace$. Hence $\mathcal{L}_{y_0}(0)\cap (-Y_+)=\lbrace 0 \rbrace$.

Next, we check the inclusion $D_S V(0,y_0)(0)\subseteq \mathcal{L}_{y_0}(0)$. Let us consider again the set-valued map $\mathcal{L}^-_{y_0}$ defined by $\mathcal{L}^-_{y_0}(z):=\mathcal{L}_{y_0}(-z)$, $\forall z \in Z$. We will prove the more general inclusion $ \text{Graph(}D_S V(0,y_0))\subseteq \text{Graph(}\mathcal{L}^-_{y_0})$. We remind the formula $\text{Graph(}\mathcal{L}^-_{y_0})=\mbox{cl}(\mbox{cone} (\mbox{Graph}(V+Y_+)-(0, y_0)))$ proved in the proof of Theorem \ref{TeorConting}. By (\ref{Inclusion_Derivada_S}), we have $\text{Graph(}D_S (V+Y_+)(0,y_0))\subseteq \text{Graph(}\mathcal{L}^-_{y_0})$. 
Then, since $\text{Graph}(V)\subset \text{Graph}(V+Y_+)$, we get that $\text{Graph(}D_S V(0,y_0))\subset \text{Graph(}D_S (V+Y_+)(0,y_0)$. Thus we have $\text{Graph(}D_S V(0,y_0))\subseteq \text{Graph(}\mathcal{L}^-_{y_0})$.

Finally, combining $D_S V(0,y_0)(0)\subseteq \mathcal{L}_{y_0}(0)$ and $\mathcal{L}_{y_0}(0)\cap (-Y_+)=\lbrace 0 \rbrace$ we obtain $D_S V(0,y_0)(0) \cap (-Y_+)=\lbrace 0 \rbrace$.
\end{proof}

Properly efficient points were introduced for two main reasons: to eliminate anomalous optimal solutions and to propose some scalar problems whose solutions provide most of the efficient points.  Next, we recall some notions regarding proper efficiency.  We say that a point $\bar{y} \in A$ is a positive properly efficient point of a set $A\subset Y$, written $\bar{y}\in$ Pos($A$), if there exists $f \in Y^*$ such that $f(y)>0$, $\forall y \in Y_+\setminus\{0\}$, and $f(\bar{y})\leq f(y)$, $\forall y \in A$. We say that a point $\bar{y} \in A$ is a Henig global properly efficient point of a set $A\subset Y$, written $\bar{y}\in $ GHe($A$), if there exists a pointed cone $\mathcal{K}$ such that  $Y_+\setminus \{0\}\subset \mbox{int}(\mathcal{K})$ and $(A-\bar{y})\cap(-\mbox{int}(\mathcal{K})) = \emptyset$.  We say that a point $\bar{y} \in A$ is a Henig properly efficient point  of a set $A\subset Y$, written $\bar{y}\in$ He($A$), if for some base $\Theta$ of $Y_+$ there is $\varepsilon > 0$ such that $\mbox{cl}(\mbox{cone} (A-\bar{y}))\cap(- \mbox{cl}(\mbox{cone} (\Theta+\varepsilon B_{Y}))) = \{ 0 \}$. We say that a point $\bar{y} \in A$ is a super efficient point of a set $A\subset Y$, written $\bar{y}\in$ SE($A$), if there is a scalar $\rho > 0$ such that $\mbox{cl}(\mbox{cone} (A-\bar{y}))\cap (B_{Y}-Y_+) \subset\rho B_{Y}$.  Regarding the program $(P(z))$, we say that a pair $(x_z,y_z)\in S(z)\times V(z)$ with $y_z \in F(x_z) $   is a  positive properly (resp.  Hening global  properly, Hening  properly, super efficient) minimizer of $(P(z))$ if $y_z\in$ Pos$(V(z))$ (resp. $y_z\in$ GHe$(V(z))$, $y_z\in$ He$(V(z))$, $y_z\in$ SE$(V(z))$). Then $y_z$ is said to be a positive properly (resp.  Hening global  properly, Hening  properly, super efficient) minimal point of $(P(z))$.  By \cite[Proposition 21.4]{Ha2010} we have the  inclusions $\mbox{Pos}(V(z))\subset \mbox{GHe}(V(z))$, $\mbox{SE}(V(z))\subset\mbox{GHe}(V(z))$,  and $\mbox{SE}(V(z))\subset \mbox{He}(V(z))$.

In the statement of the following result, we assume domination property which is usually required in  sensitivity analysis. We say that the set-valued map $V:Z \rightrightarrows Y$ is $Y_+$-dominated by  the set-valued map $M:Z \rightrightarrows Y$  near $0 \in Z$ if there exists a neighbourhood $V_0$ of $0$ such that $V(z)\subseteq M(z)+Y_+$, for every $z \in V_0\subset Z$. 
\begin{theorem} \label{thsensi}
Let $y_0\in V(0)$ be a minimal point of $(P(0))$ and $\mathcal{L}_{y_0}$ be the Lagrange process of $(P(0))$ at $y_0$. Assume the following: there exists $x_1 \in \Omega$ for which $G(x_1) \cap \left(  - \text{int}(Z_+) \right)\neq \emptyset$,  the set $Y_+ \cap S_Y $ is compact, and $V$ is $Y_+$-dominated by $M$ near $0$. If any of the following conditions holds true: (a) Graph$(\mathcal{L}_{y_0})$ has a bounded base,  (b) $Y_+ \setminus \{ 0 \}$ is open, (c) $y_0 \in \mbox{GHe}(V(0))$,  or (d) $y_0 \in \mbox{He}(V(0))$, then we have 
\begin{equation}\label{rdo_sens} 
\text{Min} \,  DM(0,y_0)(z) = \text{Min} \, \mathcal{L}_{y_0}(-z), \, \forall z \in Z.
\end{equation} 
\end{theorem}
\begin{proof}
Assume that at least one of conditions (a)-(d) above holds true.  By \cite[Theorem 1.1]{GC-Melguizo2019}, $\mathcal{L}_{y_0}$ is a Lagrange multiplier of $(P(0))$ at $y_0$. Now,
let us pick an arbitrary $z \in Z$. From Theorem \ref{TeorConting}, we have the equality $\text{Min} \, \mathcal{L}_{y_0}(-z)= \text{Min} \, D (V+Y_+)(0,y_0)(z)$.  By Lemma  \ref{sens1}, the last set equals $\text{Min} \,DV(0,y_0)(z)$. Now \cite[Theorem  13.1.4  $(ii)$]{Zalinescu2015} gives $\text{Min} \,  DV(0,y_0)(z)   = \text{Min} \,  DM(0,y_0)(z)$, yielding the equality $\text{Min} \, \mathcal{L}_{y_0}(-z)= \text{Min} \,  DM(0,y_0)(z)$.
\end{proof}

\begin{remark}\label{Remark_sensibilidad} 
The sensitivity of the program $(P(0))$ is closely related to the solutions of the dual program $(D(0))$. Indeed, equality (\ref{rdo_sens}) expresses the first-order sensitivity of $(P(0))$ at $y_0$ in terms of the Lagrange process $\mathcal{L}_{y_0}$. Besides, (\ref{Formula_defi_Gamma_y0}), (\ref{Ec_defi_LagrangeProcess}), and the equality $\mbox{Graph}(-S_{(z^*,y^*)})= \mbox{cl} \left( \cup_{ \Delta \in \Lambda_{(z^*,y^*)}}{\mbox{Graph}(\Delta)} \right)$ show the strong dependence of the Lagrange process $\mathcal{L}_{y_0}$ on the set  $\Gamma_{y_0}$ of solutions of the dual problem $(D(0))$.
\end{remark}

Next, we see how our approach contains the conventional scalar case.
\begin{remark} \label{remarkscalar} 
Let us take  $Y=\mathbb{R}$ and $Y_+=\mathbb{R}_+$; and assume that $F$ and $G$ in $(P(z))$ are single-valued maps. Then  $(P(0))$ becomes a conventional  scalar convex program and  $M$ becomes the marginal function. If we suppose  $M$ Fr\'echet differentiable, then $ M'(0;z)=DM(0,y_0)(z) $; where $M'(0;\cdot)$ stands for the Fr\'echet differential  of $M$ at $0$. On the other hand,  $\mathcal{L}_{y_0}(z)=\ell_0(z)+\mathbb{R}_+$ for every $z \in Z $; where  $\ell_0$ stands for the ``classical" scalar Lagrange multiplier of $(P(0))$.   Theorem \ref{thsensi} yields the following chain of equalities $M'(0,z)= \text{Min} M'(0,z)= \text{Min} DM(0,y_0)(z)=\text{Min}  \mathcal{L}_{y_0}(-z)=\text{Min} (\ell_0(-z)+\mathbb{R}_+)=\ell_0(-z) =-\ell_0(z)$, for every $z \in Z$. Those recover a well-known relationship between the sensitivity of a program and its Lagrange multiplier (see for instance \cite[Section 8.5, p. 221]{Luenberger1969}).
\end{remark}

The property of boundedness of the unit ball  is essential in the proof of Lemma \ref{Lemma_cono}. However,  the subsequent proposition states the following.  For any weak neighbourhood of the origin $W$,  the set ${\mbox{cl}(\mbox{cone}}(x_0+W))\cap \mbox{Ker}T$ contains a subspace of finite codimension; in  contraposition to the inclusion $\mbox{cl}(\mbox{cone}(x_0+\delta B))\setminus\{0\}\subset \{x \in X\colon T(x)>0\}$, which we find in the statement of Lemma~\ref{Lemma_cono}. This fact does not allow us to extend our argument in the proof of Lemma~\ref{Lemma_cono} to locally convex spaces. 

\begin{proposition}\label{Prop_no_elc}
Let $X$ be a normed space, $T \in X^*$, and $x_0 \in X$ such that $T(x_0)>0$. Fix $n \geq 1$, $\epsilon>0$, $T_i \in X^*$, $\forall i \in \{1,\ldots,n\}$, and consider $W=\cap_{i=1}^n\{x \in X \colon T_i(x)\leq \epsilon\}$ such that $\{x \in X \colon T(x) \leq 0\}\cap (x_0+W)=\emptyset$. 
Then $\mbox{cl}(\mbox{cone} (x_0+W))= \cap_{i=1}^n\{x \in X \colon T_i(x)\leq 0\}\cup \mbox{cone} (x_0+W)$. As a consequence, the set,  $\mbox{cl}(\mbox{cone}(x_0+W))\cap \mbox{Ker}T$, contains a subspace of finite codimension. 
\end{proposition}
\begin{proof}
We choose $y \in \mbox{cl}(\mbox{cone} (x_0+W))$, $y \not =0$, and $y \not \in \mbox{cone} (x_0+W)$.We will check that $y \in \cap_{i=1}^n\{x \in X \colon T_i(x)\leq 0\}$. Then, $y =\lim_{\alpha} \lambda_{\alpha}(x_0+w_{\alpha})$ for some nets $(\lambda_{\alpha})_{\alpha}\subset \mathbb{R}_{++}$ and $(w_{\alpha})_{\alpha}\subset W$. As in the proof of Lemma \ref{Lemma_cono}, the conditions $0 \not \in x_0+W$ and $y \not \in \mbox{cone} (x_0+W)$ assure that  $\lim_{\alpha} \lambda_{\alpha}=0$. Then $y =\lim_{\alpha} \lambda_{\alpha}w_{\alpha}$. Thus $\parallel y \parallel  =\lim_{\alpha} \lambda_{\alpha}\parallel w_{\alpha} \parallel$, which yields $\lim_{\alpha} \parallel w_{\alpha} \parallel = + \infty$. Therefore, it is not restrictive to assume that $\parallel w_{\alpha} \parallel>0$ for every $\alpha$. We have the equality $y/\parallel y \parallel=\lim_{\alpha}  w_{\alpha} /\parallel w_{\alpha} \parallel$. Now, fix any $i \in \{1,\ldots,n\}$. Then $T_i(y/\parallel y \parallel)=\lim_{\alpha}T_i(w_{\alpha} /\parallel w_{\alpha} \parallel)\leq \lim_{\alpha}\epsilon/\parallel w_{\alpha} \parallel=0$. As $i$ was arbitrarily taken, we have $y \in \cap_{i=1}^n\{x \in X \colon T_i(x)\leq 0\}$. For the reverse inclusion we consider some $y \in \cap_{i=1}^n\{x \in X \colon T_i(x)\leq 0\}$ and fix $i \in \{1,\ldots,n\}$. For any $m \geq 1$, we have $T_i(x_0+my) \leq T_i(x_0)<T_i(x_0)+\epsilon$. Then $x_0+my \in x_0+W$, for every $m \geq 1$. Now,  the equality $y =\lim_m (1/m)(x_0+my)$ yields $y \in \mbox{cl}(\mbox{cone} (x_0+W))$. Finally,  since the space $\cap_{i=1}^n \mbox{Ker} T_i \cap  \mbox{Ker}T$ has finite codimension and it is contained in $\mbox{cl}(\mbox{cone}(x_0+W))\cap \mbox{Ker}T$, the proof is over.
\end{proof}



\section{Conclusions} \label{Conclusions}

In this work, we provide a new set-valued extension of the classical Lagrange multiplier theorem for a constrained convex set-valued optimization problem.  In previous approaches, the Lagrange multipliers were usually linear continuous operators, but in this manuscript, the Lagrange multipliers are pointed closed convex processes. We set a dual program whose dual variables are also pointed closed convex processes. The property of being pointed is essential for the main results in the paper.  We prove that the Lagrange multipliers are solutions of the dual program.  We check that the sensitivity of the problem is closely related to the set of solutions of the dual program.  The arguments followed in this work are based on geometric principles and similar to those used in the scalar case.

We present some issues for further research. Since each minimal point of the primal program has associated many Lagrangian multipliers, it is of interest to determine which may be the most appropriate.  So, we can pose a first question: how does the graph of a Lagrangian multiplier $\Delta$ influence the program ($P[\Delta]$)  and its corresponding solutions? On the other hand,  pointed closed convex processes in set-valued analysis can be interpreted as the natural analogues to sublinear functions in the scalar case.  That leads us to the second question: is it possible to extend the approach developed in this work to non-convex settings? Finally, Proposition \ref{Prop_no_elc} shows that our arguments do not allow us to generalize Lemma \ref{Lemma_cono}  to locally convex spaces. That motivates our last question: can Lemma \ref{Lemma_cono}  be avoided so that we can extend the results of Sections 3 and 4 to locally convex spaces? 
\section*{Acknowledgement}
We thank the referees for their very valuable suggestions which have helped us a lot to improve the manuscript.

\section*{Funding}
The authors have been supported by project MTM2017-86182-P (AEI, Spain and ERDF/FEDER, EU).

\end{document}